\documentclass[11pt, a4paper, twoside]{amsart}
\usepackage{amssymb}
\usepackage{amsfonts, graphicx, color}
\usepackage{mathrsfs}

\theoremstyle{plain}
\newtheorem{theorem}{Theorem}

\newtheorem{corollary}{Corollary}

\newtheorem{proposition}{Proposition}

\numberwithin{equation}{section}
\begin{document}
\title[Maximal estimates for the bilinear Riesz means on Heisenberg groups]{Maximal estimates for the bilinear Riesz means on Heisenberg groups}

\author{Min Wang$^{1}$}
\address{School of Science, China University of Geosciences in Beijing, Beijing,
100083, P. R. China, }
\email{wangmin09150102@163.com}
\author{Hua Zhu$^{1,*}$}
\address{Department of Basic Sciences, Beijing International Studies University, Beijing,
100024, P. R. China, }
\email{zhuhua@bisu.edu.cn}
\thanks{$1$\ These two authors contributed to this work equally and should be regarded as co-first authors.}
\thanks{$^{*}$\ Corresponding author}

\begin{abstract}
In this article, we investigate the maximal bilinear Riesz means $S^{\alpha }_{*}$
associated to the sublaplacian on the Heisenberg group. We prove that the
operator $S^{\alpha }_{*}$ is bounded from $L^{p_{1}}\times L^{p_{2}}$ into $%
L^{p}$ for $2\leq p_{1}, p_{2}\leq \infty $ and $1/p=1/p_{1}+1/p_{2}$ when $%
\alpha $ is large than a suitable smoothness index $\alpha (p_{1},p_{2})$. For obtaining a lower index $\alpha (p_{1},p_{2})$, we define two important auxiliary operators and investigate their $L^{p}$ estimates,which play a key role in our proof.
\end{abstract}

\maketitle

\allowdisplaybreaks

\section{Introduction}

A classical problem in Fourier analysis is to make precise the sense in
which the Fourier inversion formula%
\begin{equation*}
f(x)=\int_{%
\mathbb{R}
^{n}}\widehat{f}(\xi )e^{2\pi ix\cdot \xi }d\xi
\end{equation*}%
holds when $f$ is a function on $%
\mathbb{R}
^{n}$. A natural way of formulating this identity is in term of a
summability method. For instance, one may consider the convergence of the
Bochner-Riesz means defined by
\begin{equation*}
B_{R}^{\delta }f(x)=\int_{\left\vert \xi \right\vert <R}\widehat{f}(\xi
)\left( 1-\frac{\left\vert \xi \right\vert ^{2}}{R^{2}}\right) ^{\delta
}e^{2\pi ix\cdot \xi }d\xi
\end{equation*}%
as $R\rightarrow \infty $ for some suitable $\delta $. The almost everywhere
convergence of the Bochner-Riesz is related to the maximal operator $B_{\ast
}^{\delta }=\sup_{R>0}\left\vert B_{R}^{\delta }\right\vert $. When $n=1$,
Hunt showed that if $\delta =0$ and $f\in L^{p}(%
\mathbb{R}
)$, $1<p<\infty $, $B_{R}^{\delta }f$ converges to $f$ almost everywhere as $%
R\rightarrow \infty $. When $n\geq 2$ and \thinspace $p\geq 2$, Carbery,
Rubio de Francia and Vega \cite{Carbery} showed that for all $f\in L^{p}(%
\mathbb{R}
^{n})$ with $2\leq p\leq 2n/(n-1-2\delta )$, $B_{R}^{\delta }f$ converges to
$f$ almost everywhere as $R\rightarrow \infty $ by using the weighted $L^{2}$%
-estimates of the the maximal operator $B_{\ast }^{\delta
}=\sup_{R>0}\left\vert B_{R}^{\delta }\right\vert $. When $n=2$ and $1<p<2$,
Tao \cite{Tao} proved that if $\delta >\max \left\{ \frac{3}{4p}-\frac{3}{8},%
\frac{7}{6p}-\frac{2}{3}\right\} $, $B_{\ast }^{\delta }$ is bounded from $%
L^{p}$ into $L^{p,\infty }$, i.e. $B_{R}^{\delta }f$ converges to $f$ almost
everywhere as $R\rightarrow \infty .$

The bilinear Bochner-Riesz means is defined by
\begin{equation*}
B_{R}^{\alpha }(f,g)(x)=\int_{\left\vert \xi \right\vert <R}\widehat{f}(\xi )%
\widehat{g}(\eta )\left( 1-\frac{\left\vert \xi \right\vert ^{2}+\left\vert
\eta \right\vert ^{2}}{R^{2}}\right) ^{\alpha }e^{2\pi ix\cdot (\xi +\eta
)}d\xi d\eta \text{.}
\end{equation*}%
Its almost everywhere convergence is depended on the $L^{p_{1}}\times
L^{p_{2}}\rightarrow L^{p}$ boundedness of the maximal oprator $B_{\ast
}^{\alpha }=\sup_{R>0}\left\vert B_{R}^{\alpha }\right\vert $. Grafakos, He
and Honz\'{\i}k \cite{GHH} showed that $B_{\ast }^{\alpha }$ is bounded from
$L^{2}(%
\mathbb{R}
^{n})\times L^{2}(%
\mathbb{R}
^{n})$ into $L^{1}(%
\mathbb{R}
^{n})$ if $\alpha >\frac{2n+3}{4}$. Jeong and Lee \cite{J-L} gave a
comprehensive study on this problem when $n\geq 2$ and $2\leq
p_{1},p_{2}\leq \infty $, $1/p=1/p_{1}+1/p_{2}$. Inspired by their work, we shall investigated the $L^{p_{1}}\times
L^{p_{2}}\rightarrow L^{p}$ boundedness of the maximal bilinear Riesz means
on the Heisenberg group.

Strichartz \cite{Strich, Strich2} developed the harmonic analysis on the
Heisenberg group as the spectral theory of the sublaplacian. One may define
the Riesz means in terms of the spectral decomposition of the sublaplacian.
Let
\begin{equation*}
\mathcal{L}f=\int_{0}^{\infty }\lambda P_{\lambda }f\,d\mu (\lambda )
\end{equation*}%
be the spectral decomposition of the sublaplacian $\mathcal{L}$. The Riesz
means associated to the sublaplacian $\mathcal{L}$ is defined by
\begin{equation*}
S_{r}^{\delta }f=\int_{0}^{\infty }\left( 1-r\lambda \right) _{+}^{\delta
}P_{\lambda }fd\mu (\lambda ).
\end{equation*}%
Gorges and M\"{u}ller \cite{Gorges} investigate the almost everywhere
convergence of the Riesz means and obtained a similar result to that in \cite%
{Carbery}. They showed that $S_{r}^{\delta }f\rightarrow f$ as $r\rightarrow
0$ for $\delta >0$ and $f\in L^{p}(\mathbb{H}^{n})$ provided that $\frac{Q-1%
}{Q}\left( \frac{1}{2}-\frac{\delta }{D-1}\right) <\frac{1}{p}\leq \frac{1}{2%
}$.

The bilinear Riesz means associated to the sublaplacian $\mathcal{L}$ on the
Heisenberg is defined by
\begin{equation*}
S_{r}^{\alpha }(f,g)=\int_{0}^{\infty }\int_{0}^{\infty }\left( 1-r(\lambda
_{1}+\lambda _{2})\right) _{+}^{\alpha }P_{\lambda _{1}}fP_{\lambda
_{2}}g\,d\mu (\lambda _{1})d\mu (\lambda _{2}).
\end{equation*}%
The corresponding maximal operator is denoted by $S_{\ast }^{\alpha
}=\sup_{r>0}\left\vert S_{r}^{\alpha }\right\vert $. As same as the
Euclidean case, we hope to obtain the smooth indices $\alpha (p_{1},p_{2})$
as low as possible so that $S_{\ast }^{\alpha }$ is bounded from $L^{p_{1}}(%
\mathbb{H}^{n})\times L^{p_{2}}(\mathbb{H}^{n})$ into $L^{p}(\mathbb{H}^{n})$
when $\alpha >\alpha (p_{1},p_{2})$.

\section{Preliminaries}

First we recall some basic facts about the Heisenberg group. These facts are
familiar and easy to find in many references. Let $\mathbb{H}^{n}$ denote
the Heisenberg group whose underlying manifold is $\mathbb{C}^{n}\times
\mathbb{R}$ and the group law is given by
\begin{equation*}
(z,t)(w,s)=(z+w,t+s+\frac{1}{2} \text{Im}(z\cdot \overline{w})).
\end{equation*}
The Haar measure on $\mathbb{H}^{n}$ coincides with the Lebesgue measure on $%
\mathbb{C}^{n}\times\mathbb{R}$. A homogeneous structure on $\mathbb{H}^{n}$
is given by the non-isotropic dilations $\delta_{r}(z,t)=(rz,r^{2}t)$. We
define a homogeneous norm on $\mathbb{H}^{n}$ by
\begin{equation*}
\left\vert x \right\vert = \Big( \frac{1}{16} \left\vert z\right\vert ^{4}+
t^{2} \Big)^{\frac{1}{4}}, \quad x =(z,t)\in \mathbb{H}^{n}.
\end{equation*}
This norm satisfies the triangle inequality and leads to a left-invariant
distance $d(x,y) = \left\vert x^{-1}y \right\vert$. The ball of radius $r$
centered at $x$ is
\begin{equation*}
B(x,r)= \{ y\in \mathbb{H}^{n}: \left\vert x^{-1}y \right\vert <r \}.
\end{equation*}
The Haar measure $dx$ satisfies $d\delta _{r}(x)=r^{Q}dx$ where $Q=2n+2$ is
the homogeneous dimension of $\mathbb{H}^{n}$. If $f$ and $g$ are functions
on $\mathbb{H}^{n}$, their convolution is defined by
\begin{equation*}
(f \ast g)(x)= \int_{\mathbb{H}^{n}} f\left( xy^{-1}\right)g(y)\, dy,\quad
x,y \in \mathbb{H}^{n}.
\end{equation*}
For each $\lambda \in \mathbb{R}^{\ast}$ and $f\in \mathscr{S}(\mathbb{H}%
^{n})$, the inverse Fourier transform of $f$ in variable $t$ is defined by
\begin{equation*}
f^{\lambda }(z)=\int_{-\infty }^{\infty }e^{i\lambda t}f(z,t)\, dt.
\end{equation*}
An easy calculation shows that
\begin{equation*}
(f\ast g)^{\lambda }(z)=\int_{\mathbb{C}^{n}}f^{\lambda }(z-\omega
)g^{\lambda } (\omega )e^{\frac{i}{2}\lambda \mathrm{{Im}(z\cdot \overline{
\omega })}}\, d\omega ,\quad z,\omega \in \mathbb{C}^{n}.
\end{equation*}
Thus, we are led to the convolution of the form
\begin{equation*}
f\ast _{\lambda }g=\int_{\mathbb{C}^{n}}f(z-\omega )g(\omega ) e^{\frac{i}{2}
\lambda \mathrm{{Im}(z\cdot \overline{\omega })}}\, d\omega ,
\end{equation*}
which are called the $\lambda $-twisted convolution.

The sublaplacian $\mathcal{L}$ is defined by
\begin{equation*}
\mathcal{L}= -\sum_{j=1}^{n}(X_{j}^{2}+Y_{j}^{2})
\end{equation*}
where
\begin{eqnarray*}
X_{j} &=&\frac{\partial }{\partial x_{j}}+\frac{1}{2}y_{j}\frac{\partial }{
\partial t}, \quad j=1,2,\cdots ,n, \\
Y_{j} &=&\frac{\partial }{\partial y_{j}}-\frac{1}{2}x_{j}\frac{\partial }{
\partial t}, \quad j=1,2,\cdots ,n,
\end{eqnarray*}
are left invariant vector fields on $\mathbb{H}^{n}$. Up to a constant
multiple, $\mathcal{L}$ is the unique left invariant, rotation invariant
differential operator that is homogeneous of degree two. Therefore, it is
regarded as the counterpart of the Laplacian on $\mathbb{R}^{n}$. The
sublaplacian $\mathcal{L}$ is a positive and essentially self-adjoint
operator. In the following, we state the spectral decomposition of $\mathcal{%
L}$ (cf. \cite{Thang}).

Let $\varphi _{k}$ be the Laguerre functions on $\mathbb{C}^{n}$ given by
\begin{equation*}
\varphi _{k}(z)=L_{k}^{n-1}\big(\frac{1}{2}\left\vert z\right\vert ^{2}\big)%
e^{-\frac{1}{4}\left\vert z\right\vert ^{2}},
\end{equation*}%
where $L_{k}^{n-1}$ are the Laguerre polynomials of type $n-1$ defined on $%
\mathbb{R}$ by
\begin{equation*}
L_{k}^{n-1}(t)e^{-t}t^{n-1}=\frac{1}{k!}\left( \frac{d}{dt}\right)
^{k}(e^{-t}t^{k+n-1}).
\end{equation*}%
Define functions
\begin{equation*}
e_{k}^{\lambda }(z,t)=e^{-i\lambda t}\varphi _{k}^{\lambda }(z)=e^{-i\lambda
t}\varphi _{k}(\sqrt{\left\vert \lambda \right\vert }z),\quad \lambda \in
\mathbb{R}^{\ast }.
\end{equation*}%
For $f\in L^{2}(\mathbb{H}^{n})$, we have the expansion
\begin{equation}
f(z,t)=\sum_{k=0}^{\infty }\int_{-\infty }^{\infty }f\ast e_{k}^{\lambda
}(z,t)\,d\mu (\lambda )  \label{expansion}
\end{equation}%
where $d\mu (\lambda )=(2\pi )^{-n-1}\left\vert \lambda \right\vert
^{n}d\lambda $ is the Plancherel measure for $\mathbb{H}^{n}$. Each $f\ast
e_{k}^{\lambda }$ is the eigenfunction of $\mathcal{L}$ with eigenvalue $%
(2k+n)\left\vert \lambda \right\vert $. We also have the Plancherel formula
\begin{equation}
\left\Vert f\right\Vert _{2}^{2}=(2\pi )^{-2n-1}\sum_{k=0}^{\infty
}\int_{-\infty }^{\infty }\int_{\mathbb{C}^{n}}\left\vert f^{\lambda }\ast
_{\lambda }\varphi _{k}^{\lambda }(z)\right\vert ^{2}\lambda
^{2n}\,dzd\lambda .  \label{Plancherel}
\end{equation}%
Defining
\begin{equation*}
\widetilde{e}_{k}^{\lambda }(z,t)=e_{k}^{\frac{\lambda }{2k+n}}(z,t),
\end{equation*}%
we can rewrite the decomposition (\ref{expansion}) as
\begin{equation*}
f(z,t)=\int_{-\infty }^{\infty }\sum_{k=0}^{\infty }(2k+n)^{-n-1}f\ast
\widetilde{e}_{k}^{\lambda }(z,t)\,d\mu (\lambda ).
\end{equation*}%
Let
\begin{equation*}
P_{\lambda }f(z,t)=\sum_{k=0}^{\infty }(2k+n)^{-n-1}f\ast (\widetilde{e}%
_{k}^{\lambda }+\widetilde{e}_{k}^{-\lambda })(z,t).
\end{equation*}%
Then (\ref{expansion}) can be written as
\begin{equation*}
f(z,t)=\int_{0}^{\infty }P_{\lambda }f(z,t)\,d\mu (\lambda ).
\end{equation*}%
It is clear that $P_{\lambda }f$ is an eigenfunction of the $\mathcal{L}$
with eigenvalue $\lambda $ and we have the spectral decomposition
\begin{equation*}
\mathcal{L}f=\int_{0}^{\infty }\lambda P_{\lambda }f\,d\mu (\lambda ).
\end{equation*}

Define the bilinear Riesz means associated to the sublaplacian $\mathcal{L}$
for $f,g\in \mathscr{S}(\mathbb{H}^{n})$ by
\begin{equation*}
S_{r}^{\alpha }(f,g)=\int_{0}^{\infty }\int_{0}^{\infty }\left( 1-r\left(
\lambda _{1}+\lambda _{2}\right) \right) _{+}^{\alpha }P_{\lambda
_{1}}fP_{\lambda _{2}}g\,d\mu (\lambda _{1})d\mu (\lambda _{2}).
\end{equation*}%
The corresponding maximal operator is defined by
\begin{equation*}
S_{\ast }^{\alpha }(f,g)(x)=\sup_{r>0}\left\vert S_{r}^{\alpha
}(f,g)(x)\right\vert .
\end{equation*}

\section{$L^{p}$-estimate for auxiliary multiplier operator}

In this Section, we shall define two auxiliary operators and investigate
their $L^{p}(l^\infty)$ estimates, which play a key role in the proof of
our main Theorem \ref{mainTh}.

Let $I=[-1,1]$ and consider a class of smooth function
\begin{equation*}
C^{N}(I)=\left\{ \varphi :\text{supp}\varphi \subset I\text{, }\left\Vert
\varphi \right\Vert _{C^{N}(%
\mathbb{R}
)}=\max_{0\leq n\leq N}\left\Vert \frac{d^{n}}{dt^{n}}\varphi \right\Vert
_{L^{\infty }(%
\mathbb{R}
)}\leq 1\right\} .
\end{equation*}%
For $\varphi \in C^{N}(I)$ and $\rho ,\delta ,r>0$, we define the multiplier
operator
\begin{equation}
F_{\rho ,\delta ,r}^{\varphi }f(x)=\int_{0}^{\infty }\varphi \left( \frac{%
\rho -r\lambda }{\delta }\right) P_{\lambda }fd\mu (\lambda )\text{, \ }f\in %
\mathscr{S}(\mathbb{H}^{n}),  \label{squarefunction1}
\end{equation}%
and define its corresponding operator for $k\in
\mathbb{Z}
$ by
\begin{equation}
D_{\delta ,k}^{\varphi }f(x)=\left( \sum_{\rho \in \delta
\mathbb{Z}
\cap \lbrack 0,2]}\int_{1}^{2}\left\vert F_{\rho ,\delta ,2^{k}r}^{\varphi
}f(x)\right\vert ^{2}dr\right) ^{1/2}\text{, }\text{ }f\in \mathscr{S}(%
\mathbb{H}^{n}).  \label{squarefunction2}
\end{equation}%
Write $F_{\rho ,\delta }^{\varphi }=F_{\rho ,\delta ,1}^{\varphi }$ for
simplicity. Since that $P_{\lambda }$ is a convolution operator, then
\begin{equation*}
F_{\rho ,\delta ,r}^{\varphi }f(x)=\int_{\mathbb{H}^{n}}f(x\omega
^{-1})K_{\rho ,\delta ,r}^{\varphi }(\omega )d\omega
\end{equation*}%
where the kernel $K_{\rho ,\delta ,r}^{\varphi }$ is given by
\begin{equation*}
K_{\rho ,\delta ,r}^{\varphi }(\omega )=\sum_{k=0}^{\infty
}(2k+n)^{-n-1}\int_{\mathbb{R} }\varphi \left( \frac{\rho -r\left\vert
\lambda \right\vert }{\delta }\right) \widetilde{e}_{k}^{\lambda }(\omega
)d\mu (\lambda ).
\end{equation*}%
Notice that for any $t>0$,
\begin{equation*}
K_{\rho ,\delta ,\frac{r}{t}}^{\varphi }(\omega )=t^{\frac{Q}{2}}K_{\rho
,\delta ,r}^{\varphi }\left( \delta _{\sqrt{t}}\omega \right) .
\end{equation*}%
It is easy to verify that
\begin{equation}
F_{t\rho ,t\delta ,r}^{\varphi }f(x)=F_{\rho ,\delta ,\frac{r}{t}}^{\varphi
}f(x)=F_{\rho ,\delta ,r}^{\varphi }f_{\frac{1}{\sqrt{t}}}\left( \delta _{%
\sqrt{t}}x\right) .  \label{dilation}
\end{equation}%
where $f_{s}=f(\delta _{s}\cdot )$ for any $s\in
\mathbb{R}
$, $s\neq 0$. Especially, if $r=1$ and $t=1/r$, we have \
\begin{equation}
K_{\rho ,\delta ,r}^{\varphi }(\omega )=\left( \frac{1}{r}\right) ^{\frac{Q}{%
2}}K_{\rho ,\delta }^{\varphi }\left( \delta _{\frac{1}{\sqrt{r}}}\omega
\right) \text{ and }F_{\rho ,\delta ,r}^{\varphi }f(x)=F_{\rho ,\delta
}^{\varphi }f_{\sqrt{r}}\left( \delta _{\frac{1}{\sqrt{r}}}x\right) \,.
\label{dilation2}
\end{equation}

\begin{proposition}
\label{mainPro}Let $2\leq p\leq \infty $ and \thinspace $0<\delta \leq 1/4$.
Suppose that $b>\frac{1}{2}(D-1)$ where $D=2n+1$ is the topological
dimension of $\mathbb{H}^{n}$. Then, we have that
\begin{equation}
\left\Vert \left( \int_{1/2}^{1}\left\vert F_{\rho ,\delta }^{\varphi
}f\right\vert ^{2}d\rho \right) ^{1/2}\right\Vert _{L^{p}(\mathbb{H}%
^{n})}\leq C\delta ^{-\left( b-\frac{1}{2}\right) }\left\Vert f\right\Vert
_{L^{p}(\mathbb{H}^{n})}.  \label{Pro-Es1}
\end{equation}%
It follows that for any $\varepsilon>0$,
\begin{equation}
\left\Vert \sup_{k\in
\mathbb{Z}
}\left\vert D_{\delta ,k}^{\varphi }f\right\vert \right\Vert _{L^{p}(\mathbb{%
H}^{n})}=\left\Vert \sup_{k\in
\mathbb{Z}
}\left( \sum_{\rho \in \delta
\mathbb{Z}
\cap \lbrack 0,2]}\int_{1}^{2}\left\vert F_{\rho ,\delta ,2^{k}r}^{\varphi
}f\right\vert ^{2}dr\right) ^{1/2}\right\Vert _{L^{p}(\mathbb{H}^{n})}\leq
C\delta ^{-\left( b-\frac{1}{2}\right) -\varepsilon }\left\Vert f\right\Vert
_{L^{p}(\mathbb{H}^{n})}.  \label{Pro-Es2}
\end{equation}
\end{proposition}

\begin{proof}
\bigskip It is easy to prove (\ref{Pro-Es1}). By Corollary 2.6 in \cite%
{Mull}, we know that the kernel $K_{\rho ,\delta }^{\varphi }$ of
multiplier operator $F_{\rho ,\delta }^{\varphi }$ satisfies
\begin{equation*}
\int_{\mathbb{H}^{n}}\left\vert K_{\rho ,\delta }^{\varphi }(\omega
)\right\vert d\omega \leq \left\Vert \varphi \left( \frac{\rho -\cdot }{%
\delta }\right) \right\Vert _{L_{b+\frac{1}{2}}^{2}}\leq C\delta ^{-b}
\end{equation*}%
for any $b>\frac{1}{2}(D-1)$, where the Sobolev norm is defined by
\begin{equation*}
\left\Vert f\right\Vert _{L_{\alpha }^{2}}=\left( \int_{%
\mathbb{R}
}\left\vert \left\vert x\right\vert ^{\alpha }\widehat{f}(x)\right\vert
^{2}dx\right) ^{\frac{1}{2}}.
\end{equation*}%
By Young's inequality, we get that for any $1\leq p\leq \infty $,
\begin{equation*}
\left\Vert F_{\rho ,\delta }^{\varphi }f\right\Vert _{L^{p}(\mathbb{H}%
^{n})}\leq \left\Vert K_{\rho ,\delta }^{\varphi }\right\Vert _{L^{1}(%
\mathbb{H}^{n})}\left\Vert f\right\Vert _{L^{p}(\mathbb{H}^{n})}\leq C\delta
^{-b}\left\Vert f\right\Vert _{L^{p}(\mathbb{H}^{n})}.
\end{equation*}%
Then, using Minkowski's inequality for $p\geq 2$, it follows that
\begin{eqnarray*}
\left\Vert \left( \int_{\frac{1}{2}}^{1}\left\vert F_{\rho ,\delta
}^{\varphi }f\right\vert ^{2}d\rho \right) ^{1/2}\right\Vert _{L^{p}(\mathbb{%
H}^{n})} &\leq &\left( \int_{\frac{1}{2}}^{1}\left( \int_{\mathbb{H}%
^{n}}\left\vert F_{\rho ,\delta }^{\varphi }f\right\vert ^{2\cdot \frac{p}{2}%
}dx\right) ^{\frac{2}{p}}d\rho \right) ^{\frac{1}{2}} \\
&\leq &C\delta ^{\frac{1}{2}}\left\Vert F_{\rho ,\delta }^{\varphi
}f\right\Vert _{L^{p}(\mathbb{H}^{n})}\leq C\delta ^{-(b-\frac{1}{2}%
)}\left\Vert f\right\Vert _{L^{p}(\mathbb{H}^{n})}\text{.}
\end{eqnarray*}

To obtain (\ref{Pro-Es2}), we decompose interval $[0,2]$ into dyadic
subintervals as follows:%
\begin{equation*}
\lbrack 0,2]=[0,4\delta ]\cup \lbrack 4\delta ,2]\text{,}\quad \lbrack
4\delta ,2]=\bigcup_{j=-1}^{j_{0}}I_{j}=\bigcup_{j=-1}^{j_{0}}[4\delta ,2]\cap
\lbrack 2^{-j-1},2^{-j}]
\end{equation*}%
where $j_{0}$ is the smallest integer satisfying $2^{-j_{0}-1}\leq 4\delta $%
. Then, by the triangle inequality, we have that
\begin{eqnarray*}
&&\left\Vert \sup_{k\in
\mathbb{Z}
}\left\vert D_{\delta ,k}^{\varphi }f(x)\right\vert \right\Vert _{L^{p}(%
\mathbb{H}^{n})} \\
&=&\left\Vert \sup_{k\in
\mathbb{Z}
}\left( \sum_{\rho \in \delta
\mathbb{Z}
\cap \lbrack 0,2]}\int_{1}^{2}\left\vert F_{\rho ,\delta ,2^{k}r}^{\varphi
}f(x)\right\vert ^{2}dr\right) ^{1/2}\right\Vert _{L^{p}(\mathbb{H}^{n})} \\
&\leq &\left\Vert \left( \sup_{k\in
\mathbb{Z}
}\sum_{\rho \in \delta
\mathbb{Z}
\cap \lbrack 0,4\delta ]}\int_{1}^{2}\left\vert F_{\rho ,\delta
,2^{k}r}^{\varphi }f(x)\right\vert ^{2}dr+\sum_{j=-1}^{j_{0}}\sup_{k\in
\mathbb{Z}
}\sum_{\rho \in \delta
\mathbb{Z}
\cap I_{j}}\int_{1}^{2}\left\vert F_{\rho ,\delta ,2^{k}r}^{\varphi
}f(x)\right\vert ^{2}dr\right) ^{1/2}\right\Vert _{L^{p}(\mathbb{H}^{n})} \\
&\leq &\left\Vert \left( \sup_{k\in
\mathbb{Z}
}\int_{1}^{2}\sum_{\rho \in \delta
\mathbb{Z}
\cap \lbrack 0,4\delta ]}\left\vert F_{\rho ,\delta ,2^{k}r}^{\varphi
}f(x)\right\vert ^{2}dr\right) ^{1/2}\right\Vert _{L^{p}(\mathbb{H}^{n})} \\
&&+\sum_{j=-1}^{j_{0}}\left\Vert \left( \sup_{k\in
\mathbb{Z}
}\int_{1}^{2}\sum_{\rho \in \delta
\mathbb{Z}
\cap I_{j}}\left\vert F_{\rho ,\delta ,2^{k}r}^{\varphi }f(x)\right\vert
^{2}dr\right) ^{1/2}\right\Vert _{L^{p}(\mathbb{H}^{n})}.
\end{eqnarray*}%
Setting
\begin{equation*}
I_{j}=\left\Vert \left( \sup_{k\in
\mathbb{Z}
}\sum_{\rho \in \delta
\mathbb{Z}
\cap I_{j}}\int_{1}^{2}\left\vert F_{\rho ,\delta ,2^{k}r}^{\varphi
}f(x)\right\vert ^{2}dr\right) ^{1/2}\right\Vert _{L^{p}(\mathbb{H}^{n})},
\end{equation*}%
for $-1\leq j\leq j_{0}$ and
\begin{equation*}
II=\left\Vert \left( \sup_{k\in
\mathbb{Z}
}\sum_{\rho \in \delta
\mathbb{Z}
\cap \lbrack 0,4\delta ]}\int_{1}^{2}\left\vert F_{\rho ,\delta
,2^{k}r}^{\varphi }f(x)\right\vert ^{2}dr\right) ^{1/2}\right\Vert _{L^{p}(%
\mathbb{H}^{n})},
\end{equation*}%
it follows that
\begin{equation}
\left\Vert \sup_{k\in
\mathbb{Z}
}\left\vert D_{\delta ,k}^{\varphi }f(x)\right\vert \right\Vert _{L^{p}(%
\mathbb{H}^{n})}\leq \sum_{j=-1}^{j_{0}}I_{j}+II\text{. }  \label{Pro-dec}
\end{equation}%
By the first equality relation in (\ref{dilation}), we notice that for any $%
-1\leq j\leq j_{0}$, $j\neq 0$,
\begin{equation*}
\left\Vert \left( \sup_{k\in
\mathbb{Z}
}\sum_{\rho \in \delta
\mathbb{Z}
\cap I_{j}}\int_{1}^{2}\left\vert F_{\rho ,\delta ,2^{k}r}^{\varphi
}f(x)\right\vert ^{2}dr\right) ^{1/2}\right\Vert _{L^{p}(\mathbb{H}%
^{n})}=\left\Vert \left( \sup_{k\in
\mathbb{Z}
}\sum_{2^{j}\rho \in 2^{j}\rho \delta
\mathbb{Z}
\cap I_{0}}\int_{1}^{2}\left\vert F_{2^{j}\rho ,2^{j}\delta
,2^{k-j}r}^{\varphi }f(x)\right\vert ^{2}dr\right) ^{1/2}\right\Vert _{L^{p}(%
\mathbb{H}^{n})},
\end{equation*}%
and $2^{-j}\geq 2^{-j_{0}}>4\delta $ such that $2^{j}\delta <1/4$ for any $%
0<\delta \leq 1/4$. These imply that once $I_{0}\leq \delta ^{-(b-\frac{1}{2}%
)}\left\Vert f\right\Vert _{L^{p}(\mathbb{H}^{n})}$, then \thinspace $%
I_{j}\leq \left( 2^{j}\delta \right) ^{-(b-\frac{1}{2})}$ and
\begin{equation}
\sum_{j=-1}^{j_{0}}I_{j}\leq \sum_{j=-1}^{j_{0}}\left( 2^{j}\delta \right)
^{-(b-\frac{1}{2})}\leq C\delta ^{-\left( b-\frac{1}{2}\right) -\varepsilon
}\left\Vert f\right\Vert _{L^{p}(\mathbb{H}^{n})}  \label{Pro-Ij}
\end{equation}%
for any $\varepsilon >0$ since $j_{0}=O(\log (1/\delta ))$. Thus, to obtain (%
\ref{Pro-Es2}), it suffices to show that
\begin{equation*}
\max \{I_{0},II\}\leq \delta ^{-(b-\frac{1}{2})}\text{.}
\end{equation*}%
To estimate $I_{0}$, we consider the Littlewood-Paley projection operator $%
P_{m}$, $m\in
\mathbb{Z}
$, defined by
\begin{equation*}
P_{m}f=\int_{0}^{\infty }\beta (2^{-m}\lambda )P_{\lambda }fd\mu (\lambda ),
\end{equation*}%
where $\beta \in C_{0}^{\infty }[\frac{1}{2},2]$ satisfying $0\leq \beta
\leq 1$ and $\sum_{m\in
\mathbb{Z}
}\beta (2^{-m}t)=1$ for each $t>0$. Then, we have that
\begin{equation*}
f=\int_{0}^{\infty }P_{\lambda }fd\mu (\lambda )=\sum_{m\in
\mathbb{Z}
}\int_{0}^{\infty }\beta (2^{-m}\lambda )P_{\lambda }fd\mu (\lambda
)=\sum_{m\in
\mathbb{Z}
}P_{m}f.
\end{equation*}%
Since supp$\varphi \subset \lbrack -1,1]$, $\rho \in \lbrack 1/2,1]$, $r\in
\lbrack 1,2]$ and $0<\delta \leq 1/4$, we see that
\begin{equation*}
\varphi \left( \frac{\rho -2^{k}r\lambda }{\delta }\right) \beta
(2^{-m}\lambda )\equiv 0\ \text{expect}\ -3\leq k+m\leq 2\text{. }
\end{equation*}%
Using this, we can get that for any $k\in
\mathbb{Z}
$,
\begin{eqnarray}
F_{\rho ,\delta ,2^{k}r}^{\varphi }f &=&F_{\rho ,\delta ,2^{k}r}^{\varphi
}\left( \sum_{m\in
\mathbb{Z}
}P_{m}f\right)  \notag \\
&=&\sum_{m\in
\mathbb{Z}
}\int_{0}^{\infty }\varphi \left( \frac{\rho -2^{k}r\lambda }{\delta }%
\right) \beta (2^{-m}\lambda )P_{\lambda }fd\mu (\lambda )  \notag \\
&=&\sum_{\substack{ m\in
\mathbb{Z}
,  \\ -3\leq k+m\leq 2}}\int_{0}^{\infty }\varphi \left( \frac{\rho
-2^{k}r\lambda }{\delta }\right) \beta (2^{-m}\lambda )P_{\lambda }fd\mu
(\lambda )  \notag \\
&=&\sum_{\substack{ m\in
\mathbb{Z}
,  \\ -3\leq k+m\leq 2}}F_{\rho ,\delta ,2^{k}r}^{\varphi }\left(
P_{m}f\right) .  \label{L-P}
\end{eqnarray}%
Since that $p\geq 2$, applying (\ref{L-P}), (\ref{dilation2}) and Mikowski's
inequality, it follows that
\begin{eqnarray}
I_{0} &\leq &\left( \int_{\mathbb{H}^{n}}\left( \sup_{k\in
\mathbb{Z}
}\sum_{\rho \in \delta
\mathbb{Z}
\cap I_{0}}\int_{1}^{2}\left\vert F_{\rho ,\delta ,2^{k}r}^{\varphi
}f(x)\right\vert ^{2}dr\right) ^{\frac{p}{2}}dx\right) ^{\frac{1}{p}}  \notag
\\
&=&\left( \int_{\mathbb{H}^{n}}\left( \sup_{k\in
\mathbb{Z}
}\sum_{\rho \in \delta
\mathbb{Z}
\cap I_{0}}\int_{1}^{2}\left\vert \sum_{\substack{ m\in
\mathbb{Z}
,  \\ -3\leq k+m\leq 2}}F_{\rho ,\delta ,2^{k}r}^{\varphi }\left(
P_{m}f\right) (x)\right\vert ^{2}dr\right) ^{\frac{p}{2}}dx\right) ^{\frac{1%
}{p}}  \notag \\
&=&\left( \int_{\mathbb{H}^{n}}\left( \sup_{k\in
\mathbb{Z}
}\sum_{\rho \in \delta
\mathbb{Z}
\cap I_{0}}\int_{1}^{2}\left\vert \sum_{\substack{ m\in
\mathbb{Z}
,  \\ -3\leq k+m\leq 2}}F_{\rho ,\delta }^{\varphi }\left( P_{m}f\right) _{%
\sqrt{2^{k}r}}\left( \delta _{\frac{1}{\sqrt{2^{k}r}}}x\right) \right\vert
^{2}dr\right) ^{\frac{p}{2}}dx\right) ^{\frac{1}{p}}  \notag \\
&\leq &\sum_{k\in
\mathbb{Z}
}\sum_{\substack{ m\in
\mathbb{Z}
,  \\ -3\leq k+m\leq 2}}\left( \int_{\mathbb{H}^{n}}\left(
\int_{1}^{2}\sum_{\rho \in \delta
\mathbb{Z}
\cap I_{0}}\left\vert F_{\rho ,\delta }^{\varphi }\left( P_{m}f\right) _{%
\sqrt{2^{k}r}}\left( \delta _{\frac{1}{\sqrt{2^{k}r}}}x\right) \right\vert
^{2}dr\right) ^{\frac{p}{2}}dx\right) ^{\frac{2}{p}\cdot \frac{1}{2}}  \notag
\\
&\leq &\sum_{k\in
\mathbb{Z}
}\sum_{\substack{ m\in
\mathbb{Z}
,  \\ -3\leq k+m\leq 2}}\left( \int_{1}^{2}\left( \sqrt{2^{k}r}\right) ^{%
\frac{2Q}{p}}\left\Vert \sum_{\rho \in \delta
\mathbb{Z}
\cap I_{0}}\left\vert F_{\rho ,\delta }^{\varphi }\left( P_{m}f\right) _{%
\sqrt{2^{k}r}}\right\vert ^{2}\right\Vert _{L^{\frac{p}{2}}(\mathbb{H}%
^{n})}dr\right) ^{\frac{1}{2}}  \notag \\
&=&\sum_{k\in
\mathbb{Z}
}\sum_{\substack{ m\in
\mathbb{Z}
,  \\ -3\leq k+m\leq 2}}\left( \int_{1}^{2}\left( \sqrt{2^{k}r}\right) ^{%
\frac{2Q}{p}}\left\Vert \left( \sum_{\rho \in \delta
\mathbb{Z}
\cap I_{0}}\left\vert F_{\rho ,\delta }^{\varphi }\left( P_{m}f\right) _{%
\sqrt{2^{k}r}}\right\vert ^{2}\right) ^{\frac{1}{2}}\right\Vert _{L^{p}(%
\mathbb{H}^{n})}^{2}dr\right) ^{\frac{1}{2}}.  \label{insert}
\end{eqnarray}%
Notice that the $L^{p}$-boundedness properties of the square function in (%
\ref{Pro-Es1}) and the discretize square function in the above are
essentially equaivalent. Hence, we have that
\begin{eqnarray*}
\left\Vert \left( \sum_{\rho \in \delta
\mathbb{Z}
\cap I_{0}}\left\vert F_{\rho ,\delta }^{\varphi }\left( P_{m}f\right) _{%
\sqrt{2^{k}r}}\right\vert ^{2}\right) ^{\frac{1}{2}}\right\Vert _{L^{p}(%
\mathbb{H}^{n})} &\leq &C\delta ^{-(b-\frac{1}{2})}\left\Vert \left(
P_{m}f\right) _{\sqrt{2^{k}r}}\right\Vert _{L^{p}(\mathbb{H}^{n})} \\
&\leq &C\delta ^{-(b-\frac{1}{2})}\left( \sqrt{2^{k}r}\right) ^{-\frac{Q}{p}%
}\left\Vert P_{m}f\right\Vert _{L^{p}(\mathbb{H}^{n})}.
\end{eqnarray*}%
Inserting this into (\ref{insert}) and using the Littlewood-Paley theorem,
we can obtain that%
\begin{eqnarray}
I_{0} &\leq &\sum_{k\in
\mathbb{Z}
}\sum_{\substack{ m\in
\mathbb{Z}
,  \\ -3\leq k+m\leq 2}}\left( \int_{1}^{2}\left( \sqrt{2^{k}r}\right) ^{%
\frac{2Q}{p}}\left\Vert \left( \sum_{\rho \in \delta
\mathbb{Z}
\cap I_{0}}\left\vert F_{\rho ,\delta }^{\varphi }\left( P_{m}f\right) _{%
\sqrt{2^{k}r}}\right\vert ^{2}\right) ^{\frac{1}{2}}\right\Vert _{L^{p}(%
\mathbb{H}^{n})}^{2}dr\right) ^{\frac{1}{2}}  \notag \\
&\leq &C\delta ^{-\left( b-\frac{1}{2}\right) }\left( \sum_{k\in
\mathbb{Z}
}\sum_{\substack{ m\in
\mathbb{Z}
,  \\ -3\leq k+m\leq 2}}\left\Vert P_{m}f\right\Vert _{L^{p}(\mathbb{H}%
^{n})}^{2}\right) ^{\frac{1}{2}}  \notag \\
&\leq &C\delta ^{-\left( b-\frac{1}{2}\right) }\left\Vert \left( \sum_{m\in
\mathbb{Z}
}\left\vert P_{m}f\right\vert ^{2}\right) ^{\frac{1}{2}}\right\Vert _{L^{p}(%
\mathbb{H}^{n})}  \notag \\
&\leq &C\delta ^{-\left( b-\frac{1}{2}\right) }\left\Vert f\right\Vert
_{L^{p}(\mathbb{H}^{n})}.  \label{Pro-result-1}
\end{eqnarray}%
Next, we consider the estimate of $II$. Notice that
\begin{equation*}
F_{\rho ,\delta }^{\varphi }f(x)=\int_{0}^{\infty }\varphi \left( \frac{\rho
-\lambda }{\delta }\right) P_{\lambda }fd\mu (\lambda )=\int_{\mathbb{H}%
^{n}}f(x\omega ^{-1})K_{\rho ,\delta }(\omega )d\omega .
\end{equation*}%
Setting $R_{t}^{l}(\omega )$ to be the kernel of the Riesz means $%
\int_{0}^{t}(1-\frac{\lambda }{t})^{l}P_{\lambda }fd\mu (\lambda )$, we see
that
\begin{equation*}
t\rightarrow R_{t}^{0}(\omega )
\end{equation*}%
is a function of bounded variation. Then, the kernel $K_{\rho ,\delta
}^{\varphi }$ can be written as%
\begin{equation*}
K_{\rho ,\delta }(\omega )=\int_{0}^{\infty }\varphi \left( \frac{\rho
-\lambda }{\delta }\right) \frac{\partial }{\partial \lambda }R_{\lambda
}^{0}(\omega )d\mu (\lambda ).
\end{equation*}%
Intergration by parts and using the identity
\begin{equation}
\frac{\partial }{\partial t}(t^{m}R_{t}^{m}(\omega
))=mt^{m-1}R_{t}^{m-1}(\omega ),  \label{identity}
\end{equation}%
where $m$ is a positive integer, we get that
\begin{equation*}
K_{\rho ,\delta }(\omega )=c_{m}\int_{0}^{\infty }\left( \partial _{\lambda
}^{2m+2}\varphi \left( \frac{\rho -\lambda }{\delta }\right) \right) \lambda
^{2m+1}R_{\lambda }^{2m+1}(\omega )d\mu (\lambda ).
\end{equation*}%
It is known that (see Theorem 2.5.3 in \cite{Thang})
\begin{equation}
\left\vert R_{\lambda }^{2m+1}(\omega )\right\vert \leq C\lambda ^{\frac{Q}{2%
}}(1+\lambda ^{\frac{1}{2}}\left\vert \omega \right\vert )^{-2m}.
\label{estimate-R}
\end{equation}%
We let $m=\frac{Q}{2}+1$ and $\varphi \in C^{N}(I)$ with $N=2m+2$. Then, we
have that
\begin{equation*}
\left\vert \partial _{\lambda }^{2m+2}\varphi \left( \frac{\rho -\lambda }{%
\delta }\right) \right\vert \leq \delta ^{-(2m+2)}.
\end{equation*}%
This together with (\ref{estimate-R}) yield that for any $\rho \in \delta
\mathbb{Z}
\cap \lbrack 0,4\delta ]$
\begin{eqnarray*}
\left\vert K_{\rho ,\delta }(\omega )\right\vert &\leq &c_{m}(1+\left\vert
\omega \right\vert )^{-2m}\delta ^{-(2m+2)}\int_{0}^{5\delta }\lambda
^{2m+1-m+\frac{Q}{2}+n}d\lambda \\
&\leq &c_{m}\delta ^{-m+n+\frac{Q}{2}}(1+\left\vert \omega \right\vert
)^{-2m}\leq c_{m}(1+\left\vert \omega \right\vert )^{-2m}.
\end{eqnarray*}%
Using (\ref{dilation2}) and Young's inequality, it follows that for any $%
r\in \lbrack 1,2]$, $k\in
\mathbb{Z}
$ and $\rho \in \delta
\mathbb{Z}
\cap \lbrack 0,4\delta ]$,
\begin{equation*}
\left\Vert F_{\rho ,\delta ,2^{k}r}^{\varphi }f\right\Vert _{L^{p}(\mathbb{H}%
^{n})}=\left\Vert F_{\rho ,\delta }^{\varphi }f\right\Vert _{L^{p}(\mathbb{H}%
^{n})}\leq c_{m}\left\Vert f\right\Vert _{L^{p}(\mathbb{H}^{n})}\int_{%
\mathbb{H}^{n}}(1+\left\vert \omega \right\vert )^{-2m}d\omega \leq
C\left\Vert f\right\Vert _{L^{p}(\mathbb{H}^{n})}.
\end{equation*}%
Hence, by Minkowski's inequality, we can get that
\begin{eqnarray}
II &=&\left\Vert \left( \sup_{k\in
\mathbb{Z}
}\sum_{\rho \in \delta
\mathbb{Z}
\cap \lbrack 0,4\delta ]}\int_{1}^{2}\left\vert F_{\rho ,\delta
,2^{k}r}^{\varphi }f(x)\right\vert ^{2}dr\right) ^{1/2}\right\Vert _{L^{p}(%
\mathbb{H}^{n})}  \notag \\
&\leq &\sup_{k\in
\mathbb{Z}
}\left( \int_{1}^{2}\left\Vert \sum_{\rho \in \delta
\mathbb{Z}
\cap \lbrack 0,4\delta ]}\left\vert F_{\rho ,\delta ,2^{k}r}^{\varphi
}f\right\vert ^{2}\right\Vert _{L^{\frac{p}{2}}(\mathbb{H}^{n})}dr\right) ^{%
\frac{1}{2}}  \notag \\
&\leq &\sup_{k\in
\mathbb{Z}
}\left( \int_{1}^{2}\sum_{\rho \in \delta
\mathbb{Z}
\cap \lbrack 0,4\delta ]}\left\Vert F_{\rho ,\delta ,2^{k}r}^{\varphi
}f\right\Vert _{L^{p}(\mathbb{H}^{n})}^{2}dr\right) ^{\frac{1}{2}}\leq
C\left\Vert f\right\Vert _{L^{p}}.  \label{Pro-result-2}
\end{eqnarray}%
Applying (\ref{Pro-dec}), (\ref{Pro-Ij}) and the above estimates (\ref%
{Pro-result-1}), (\ref{Pro-result-2}), we can conclude that for any $%
\varepsilon>0$
\begin{equation*}
\left\Vert \sup_{k\in
\mathbb{Z}
}\left\vert D_{\delta ,k}^{\varphi }f(x)\right\vert \right\Vert _{L^{p}(%
\mathbb{H}^{n})}\leq C\delta ^{-\left( b-\frac{1}{2}\right) }\left\Vert
f\right\Vert _{L^{p}(\mathbb{H}^{n})}.
\end{equation*}
The proof of (\ref{Pro-Es2}) is complete.
\end{proof}

In \cite{L-W}, we proved that for any function $m\in L^{\infty }(\mathbb{R})$
and $0\leq a<b$, the multiplier operator $T_{m}f=\int_{a}^{b}m(\lambda
)P_{\lambda }f\,d\mu (\lambda )$ is bounded from $L^{p}(\mathbb{H}^{n})$
into $L^{2}(\mathbb{H}^{n})$ for any $1\leq p\leq 2$, i.e.
\begin{equation*}
\left\Vert T_{m}f\right\Vert _{2}\leq C\left\Vert m\right\Vert _{\infty
}\left( (b-a)b^{n}\right) ^{(\frac{1}{p}-\frac{1}{2})}\left\Vert
f\right\Vert _{p}.
\end{equation*}%
Using this estimate, it is easy to check that%
\begin{equation*}
\left\Vert \left( \int_{1/2}^{1}\left\vert F_{\rho ,\delta }^{\varphi
}f(x)\right\vert ^{2}d\rho \right) ^{1/2}\right\Vert _{L^{2}(\mathbb{H}%
^{n})}\leq C\delta ^{\frac{1}{2}}\left\Vert f\right\Vert _{L^{2}(\mathbb{H}%
^{n})}.
\end{equation*}%
By the same argument of Proposition \ref{mainTh}, we have that
\begin{equation}
\left\Vert \sup_{k\in
\mathbb{Z}
}\left\vert D_{\delta ,k}^{\varphi }f\right\vert \right\Vert _{L^{2}(\mathbb{%
H}^{n})}\leq C\delta ^{\frac{1}{2}-\varepsilon }\left\Vert f\right\Vert
_{L^{2}(\mathbb{H}^{n})}\label{L2}
\end{equation}%
for any $\varepsilon>0$. Then, by interpolation between the estimates in (%
\ref{L2}) and (\ref{Pro-Es2}) for $p=\infty $, we can get the following
result:

\begin{corollary}
\label{Coro}Let $2\leq p\leq \infty $ and \thinspace $0<\delta \leq 1/4$.
Suppose that $b>\frac{1}{2}(D-1)$ where $D=2n+1$ is the topological
dimension of $\mathbb{H}^{n}$. Then,
\begin{equation*}
\left\Vert \sup_{k\in
\mathbb{Z}
}\left\vert D_{\delta ,k}^{\varphi }f\right\vert \right\Vert _{L^{p}(\mathbb{%
H}^{n})}\leq C\delta ^{-\left[ (b-\frac{1}{2})(1-\frac{2}{p})-\frac{1}{p}%
\right] -\varepsilon }\left\Vert f\right\Vert _{L^{p}(\mathbb{H}^{n})}\text{%
. }
\end{equation*}
\end{corollary}

\section{Boundedness of the maximal operator $S_{\ast }^{\protect\alpha }$}

\begin{theorem}
\label{mainTh}\bigskip Let $2\leq p_{1},p_{2}\leq \infty $ and $%
1/p=1/p_{1}+1/p_{2}$. If $\alpha >$ $D(1-\frac{1}{p})+\frac{3}{p}$, then $%
S_{\ast }^{\alpha }$ is bounded from $L^{p_{1}}(\mathbb{H}^{n})\times
L^{p_{2}}(\mathbb{H}^{n})$ into $L^{p}(\mathbb{H}^{n})$.
\end{theorem}

\begin{proof}
Fix $\alpha >0$. Let us choose $\psi \in C_{0}^{\infty }([\frac{1}{2},2])$
and $\psi _{0}\in C_{0}^{\infty }([-\frac{3}{4},\frac{3}{4}])$ such that
\begin{equation*}
(1-t)_{+}^{\alpha }=\sum_{\delta \in D}\delta ^{\alpha }\psi \left( \frac{1-t%
}{\delta }\right) +\psi _{0}(t)\text{, \ \ }0\leq t\leq 1,
\end{equation*}%
where $D=\{2^{k}:k\in
\mathbb{Z}
$ and $k\leq -2\}$. Using this, we can decompose
\begin{equation*}
S_{r }^{\alpha }=\sum_{\delta \in D}^{\infty }\delta ^{\alpha }S_{r}^{\delta
}+S_{r}^{0}
\end{equation*}%
where
\begin{equation*}
S_{r}^{\delta }(f,g)=\int_{0}^{\infty }\int_{0}^{\infty }\psi \left( \frac{%
1-r\lambda _{1}-r\lambda _{2}}{\delta }\right) P_{\lambda _{1}}fP_{\lambda
_{2}}g\,d\mu (\lambda _{1})d\mu (\lambda _{2}),
\end{equation*}%
and
\begin{equation*}
S_{r}^{0}(f,g)=\int_{0}^{\infty }\int_{0}^{\infty }\psi _{0}\left( r\lambda
_{1}+r\lambda _{2}\right) P_{\lambda _{1}}fP_{\lambda _{2}}g\,d\mu (\lambda
_{1})d\mu (\lambda _{2}).
\end{equation*}%
It follows that
\begin{equation}
S_{\ast }^{\alpha }(f,g)(x)\leq \sum_{\delta \in D}\delta ^{\alpha
}\sup_{r>0}\left\vert S_{r}^{\delta }(f,g)(x)\right\vert
+\sup_{r>0}\left\vert S_{r}^{0}(f,g)(x)\right\vert .  \label{decom1}
\end{equation}%
Since that $\psi _{0}\in C_{0}^{\infty }([-\frac{3}{4},\frac{3}{4}])$, using
Holder's inequality, it is easy to see that for any $2\leq p_{1},p_{2}\leq
\infty $ and $1/p=1/p_{1}+1/p_{2}$,
\begin{equation*}
\left\Vert S_{\ast }^{0}(f,g)\right\Vert _{L^{p}(\mathbb{H}^{n})}=\left\Vert
\sup_{r>0}\left\vert S_{r}^{0}(f,g)\right\vert \right\Vert _{L^{p}(\mathbb{H}%
^{n})}\leq \left\Vert fg\right\Vert _{L^{p}(\mathbb{H}^{n})}\leq \left\Vert
f\right\Vert _{L^{p_{1}}(\mathbb{H}^{n})}\left\Vert g\right\Vert _{L^{p_{2}}(%
\mathbb{H}^{n})}.
\end{equation*}%
Therefore, to obtain Theorem \ref{mainTh}, we have to focus on obtaining
estimates for the maximal operator
\begin{equation*}
S_{\ast }^{\delta }(f,g)(x)=\sup_{r>0}\left\vert S_{r}^{\delta
}(f,g)(x)\right\vert \quad\text{for}\quad 0<\delta \leq \frac{1}{4}.
\end{equation*}%
By the fundamental theorem of calculus, we see that $\left\vert
F(t)\right\vert \leq \left\vert F(s)\right\vert +\int_{1}^{2}\left\vert
F^{\prime }(\tau )\right\vert d\tau $ for any $t,s\in[1,2]$. This implies
that
\begin{equation*}
S_{\ast }^{\delta }(f,g)(x)=\sup_{k\in
\mathbb{Z}
}\sup_{1\leq r\leq 2}\left\vert S_{2^{k}r}^{\delta }(f,g)(x)\right\vert \leq
\sup_{k\in
\mathbb{Z}
}\int_{1}^{2}\left\vert S_{2^{k}r}^{\delta }(f,g)(x)\right\vert
dr+\sup_{k\in
\mathbb{Z}
}\int_{1}^{2}\left\vert \frac{\partial }{\partial r}S_{2^{k}r}^{\delta
}(f,g)(x)\right\vert dr.
\end{equation*}%
Since that
\begin{equation*}
\frac{\partial }{\partial r}S_{2^{k}r}^{\delta }(f,g)=\frac{-2^{k}}{\delta }%
\int_{0}^{\infty }\int_{0}^{\infty }(\lambda _{1}+\lambda _{2})\psi ^{\prime
}\left( \frac{1-2^{k}r(\lambda _{1}+\lambda _{2})}{\delta }\right)
P_{\lambda _{1}}fP_{\lambda _{2}}g\,d\mu (\lambda _{1})d\mu (\lambda _{2}),
\end{equation*}%
we can conclude that $\frac{\partial }{\partial r}S_{2^{k}r}^{\delta }(f,g)$
satisfies the same quantitative properties as $\frac{1}{\delta }%
S_{2^{k}r}^{\delta }(f,g)$ when $1\leq r\leq 2$. Hence, to estimate $S_{\ast
}^{\delta }(f,g)$, it suffices to consider the operator
\begin{equation*}
(f,g)\rightarrow \sup_{k\in
\mathbb{Z}
}\int_{1}^{2}\left\vert S_{2^{k}r}^{\delta }(f,g)(x)\right\vert dr.
\end{equation*}%
To eatimate this operator, we choose $\varphi \in C_{0}^{\infty }(I)$
satisfying $\sum_{l\in
\mathbb{Z}
}\varphi (t+l)=1$ for all $t\in
\mathbb{R}
$. Fix $0<\delta \leq 1/4$ and set $\widetilde{\delta }=\delta ^{1+\kappa }$
with some $\kappa >0$ . Then, for any $r>0$, we can write
\begin{equation*}
\psi \left( \frac{1-r\lambda _{1}-r\lambda _{2}}{\delta }\right)
=\sum_{\sigma \in \widetilde{\delta }%
\mathbb{Z}
}\sum_{\rho \in \widetilde{\delta }%
\mathbb{Z}
\cap \lbrack 0,2]}\varphi \left( \frac{\rho -r\lambda _{1}}{\widetilde{%
\delta }}\right) \varphi \left( \frac{\sigma -\rho -r\lambda _{2}}{%
\widetilde{\delta }}\right) \psi \left( \frac{1-r\lambda _{1}-r\lambda _{2}}{%
\delta }\right) .
\end{equation*}%
Since supp$\varphi \subset \lbrack -1,1]$ and supp$\psi \subset \lbrack
1/2,2]$, then
\begin{equation*}
\varphi \left( \frac{\rho -r\lambda _{1}}{\widetilde{\delta }}\right) \psi
\left( \frac{1-r\lambda _{1}-r\lambda _{2}}{\delta }\right) =0\quad \text{%
except} \quad 1-3\delta \leq r\lambda _{2}+\rho \leq 1+\delta.
\end{equation*}
It follows that
\begin{equation*}
\varphi \left( \frac{\rho -r\lambda _{1}}{\widetilde{\delta }}\right)
\varphi \left( \frac{\sigma -\rho -r\lambda _{2}}{\widetilde{\delta }}%
\right) \psi \left( \frac{1-r\lambda _{1}-r\lambda _{2}}{\delta }\right) =0%
\text{ except }\sigma \in \lbrack 1-4\delta ,1+2\delta ]).
\end{equation*}%
Thus,
\begin{eqnarray}
&&S_{r}^{\delta }(f,g)  \notag \\
&=&\sum_{\sigma \in \widetilde{\delta }%
\mathbb{Z}
\cap \lbrack 1-4\delta ,1+2\delta ]}\sum_{\rho \in \widetilde{\delta }%
\mathbb{Z}
\cap \lbrack 0,2]}\int_{0}^{\infty }\int_{0}^{\infty }\varphi \left( \frac{%
\rho -r\lambda _{1}}{\widetilde{\delta }}\right) \varphi \left( \frac{\sigma
-\rho -r\lambda _{2}}{\widetilde{\delta }}\right)  \notag \\
&&\text{ \ \ \ \ \ \ \ \ \ \ \ \ \ }\times \psi \left( \frac{1-r\lambda
_{1}-r\lambda _{2}}{\delta }\right) P_{\lambda _{1}}fP_{\lambda _{2}}g\,d\mu
(\lambda _{1})d\mu (\lambda _{2}).  \label{decom2}
\end{eqnarray}%
At the same time, by the Fourier inversion formula, we have
\begin{equation}
\psi \left( \frac{1-r\lambda _{1}-r\lambda _{2}}{\delta }\right) =\int_{%
\mathbb{R}
}\widehat{\psi }(\tau )e^{2\pi i\tau \left( \frac{1-r\lambda _{1}-r\lambda
_{2}}{\delta }\right) }d\tau =\int_{%
\mathbb{R}
}\widehat{\psi }(\tau )e^{2\pi i\tau \left( \frac{\sigma -r\lambda
_{1}-r\lambda _{2}}{\delta }\right) }e^{2\pi i\tau \left( \frac{1-\sigma }{%
\delta }\right) }d\tau .  \label{inversion}
\end{equation}
Applying Taylor's theorem for $e^{2\pi i\tau \left( \frac{\sigma -r\lambda
_{1}-r\lambda _{2}}{\delta }\right) }$, we get that for any $\rho \in
\widetilde{\delta }%
\mathbb{Z}
\cap \lbrack 0,2]$,
\begin{eqnarray}
e^{2\pi i\tau \left( \frac{\sigma -r\lambda _{1}-r\lambda _{2}}{\delta }%
\right) } &=&\sum_{N=0}^{\infty }\frac{1}{N!}\left( \frac{\tau (\sigma
-r\lambda _{1}-r\lambda _{2})}{\delta }\right) ^{N}  \notag \\
&=&\sum_{N=0}^{\infty }\frac{1}{N!}\sum_{0\leq a+b\leq N}c_{a,b}\left( \frac{%
\tau (\rho -r\lambda _{1})}{\delta }\right) ^{a}\left( \frac{\tau (\sigma
-\rho -r\lambda _{2})}{\delta }\right) ^{b}.  \label{Taylor}
\end{eqnarray}%
Putting (\ref{Taylor}) into the right hand side of (\ref{inversion}), it
follows that
\begin{eqnarray*}
&&\psi \left( \frac{1-r\lambda _{1}-r\lambda _{2}}{\delta }\right) \\
&=&\sum_{N=0}^{\infty }\frac{1}{N!}\sum_{0\leq a+b\leq N}c_{a,b}\left( \int_{%
\mathbb{R}
}\widehat{\psi }(\tau )\tau ^{a+b}e^{2\pi i\tau \left( \frac{1-\sigma }{%
\delta }\right) }d\tau \right) \left( \frac{\rho -r\lambda _{1}}{\delta }%
\right) ^{a}\left( \frac{\sigma -\rho -r\lambda _{2}}{\delta }\right) ^{b} \\
&=&\sum_{N=0}^{\infty }\frac{1}{N!}\sum_{0\leq a+b\leq N}c_{a,b}\psi
^{(a+b)}\left( \frac{1-\sigma }{\delta }\right) \left( \frac{\rho -r\lambda
_{1}}{\delta }\right) ^{a}\left( \frac{\sigma -\rho -r\lambda _{2}}{\delta }%
\right) ^{b}.
\end{eqnarray*}%
Insert this into (\ref{decom2}) and let $\varphi _{\beta }(t)$ $=t^{\beta
}\varphi (t)$. Notice that $\varphi _{\beta }\in C_{0}^{\infty }(I)$ for any
$\beta \in
\mathbb{N}
$. Then, we can obtain that

\begin{eqnarray*}
&&S_{r}^{\delta }(f,g) \\
&=&\sum_{\sigma \in \widetilde{\delta }%
\mathbb{Z}
\cap \lbrack 1-4\delta ,1+2\delta ]}\sum_{\rho \in \widetilde{\delta }%
\mathbb{Z}
\cap \lbrack 0,2]}\int_{0}^{\infty }\int_{0}^{\infty }\varphi \left( \frac{%
\rho -r\lambda _{1}}{\widetilde{\delta }}\right) \varphi \left( \frac{\sigma
-\rho -r\lambda _{2}}{\widetilde{\delta }}\right) \\
&&\times \psi \left( \frac{1-r\lambda _{1}-r\lambda _{2}}{\delta }\right)
P_{\lambda _{1}}fP_{\lambda _{2}}g\,d\mu (\lambda _{1})d\mu (\lambda _{2}) \\
&=&\sum_{\sigma \in \widetilde{\delta }%
\mathbb{Z}
\cap \lbrack 1-4\delta ,1+2\delta ]}\sum_{\rho \in \widetilde{\delta }%
\mathbb{Z}
\cap \lbrack 0,2]}\int_{0}^{\infty }\int_{0}^{\infty }\varphi \left( \frac{%
\rho -r\lambda _{1}}{\widetilde{\delta }}\right) \varphi \left( \frac{\sigma
-\rho -r\lambda _{2}}{\widetilde{\delta }}\right) \\
&&\times \sum_{N=0}^{\infty }\frac{1}{N!}\sum_{0\leq a+b\leq N}c_{a,b}\left(
\psi ^{(a+b)}\left( \frac{1-\sigma }{\delta }\right) \right) \left( \frac{%
\rho -r\lambda _{1}}{\delta }\right) ^{a}\left( \frac{\sigma -\rho -r\lambda
_{2}}{\delta }\right) ^{b}P_{\lambda _{1}}fP_{\lambda _{2}}g\,d\mu (\lambda
_{1})d\mu (\lambda _{2}) \\
&=&\sum_{N=0}^{\infty }\frac{1}{N!}\sum_{0\leq a+b\leq N}c_{a,b}\delta
^{\kappa (a+b)}\sum_{\sigma \in \widetilde{\delta }%
\mathbb{Z}
\cap \lbrack 1-4\delta ,1+2\delta ]}\left( \psi ^{(a+b)}\left( \frac{%
1-\sigma }{\delta }\right) \right) \\
&&\times \sum_{\rho \in \widetilde{\delta }%
\mathbb{Z}
\cap \lbrack 0,2]}\int_{0}^{\infty }\varphi _{a}\left( \frac{\rho -r\lambda
_{1}}{\widetilde{\delta }}\right) P_{\lambda _{1}}f\,d\mu (\lambda
_{1})\int_{0}^{\infty }\varphi _{b}\left( \frac{\sigma -\rho -r\lambda _{2}}{%
\widetilde{\delta }}\right) P_{\lambda _{2}}gd\mu (\lambda _{2}).
\end{eqnarray*}%
Set $r=2^{k}r$ for $k\in
\mathbb{Z}
$, $r\in \lbrack 1,2]$. Applying the triangle inequality, Cauchy-Schwartz'
inequality and H\"{o}lder's inequality, we get that
\begin{eqnarray}
&&\left\Vert \sup_{k\in
\mathbb{Z}
}\int_{1}^{2}\left\vert S_{2^{k}r}^{\delta }(f,g)\right\vert dr\right\Vert
_{L^{p}(\mathbb{H}^{n})}  \notag \\
&\leq &\left\Vert \sup_{k\in
\mathbb{Z}
}\int_{1}^{2}\sum_{N=0}^{\infty }\frac{1}{N!}\sum_{0\leq a+b\leq
N}c_{a,b}\delta ^{\kappa (a+b)}\sum_{\sigma \in \widetilde{\delta }%
\mathbb{Z}
\cap \lbrack 1-4\delta ,1+2\delta ]}\left\vert \psi ^{(a+b)}\left( \frac{%
1-\sigma }{\delta }\right) \right\vert \right.  \notag \\
&&\left. \times \sum_{\rho \in \widetilde{\delta }%
\mathbb{Z}
\cap \lbrack 0,2]}\left\vert F_{\rho ,\widetilde{\delta },2^{k}r}^{\varphi
_{a}}f\right\vert \left\vert F_{\sigma -\rho ,\widetilde{\delta }%
,2^{k}r}^{\varphi _{b}}g\right\vert dr\right\Vert _{L^{p}(\mathbb{H}^{n})}
\notag \\
&\leq &\sum_{N=0}^{\infty }\frac{1}{N!}\sum_{0\leq a+b\leq N}c_{a,b}\delta
^{\kappa (a+b)}\sum_{\sigma \in \widetilde{\delta }%
\mathbb{Z}
\cap \lbrack 1-4\delta ,1+2\delta ]}\left\vert \psi ^{(a+b)}\left( \frac{%
1-\sigma }{\delta }\right) \right\vert  \notag \\
&&\times \left\Vert \sup_{k\in
\mathbb{Z}
}\int_{1}^{2}\sum_{\rho \in \widetilde{\delta }%
\mathbb{Z}
\cap \lbrack 0,2]}\left\vert F_{\rho ,\widetilde{\delta },2^{k}r}^{\varphi
_{a}}f\right\vert \left\vert F_{\sigma -\rho ,\widetilde{\delta }%
,2^{k}r}^{\varphi _{b}}g\right\vert dr\right\Vert _{L^{p}(\mathbb{H}^{n})}
\notag \\
&\leq &\sum_{N=0}^{\infty }\frac{1}{N!}\sum_{0\leq a+b\leq N}c_{a,b}\delta
^{\kappa (a+b)}\sum_{\sigma \in \widetilde{\delta }%
\mathbb{Z}
\cap \lbrack 1-4\delta ,1+2\delta ]}\left\vert \psi ^{(a+b)}\left( \frac{%
1-\sigma }{\delta }\right) \right\vert  \notag \\
&&\times \left\Vert \sup_{k\in
\mathbb{Z}
}\left( \sum_{\rho \in \widetilde{\delta }%
\mathbb{Z}
\cap \lbrack 0,2]}\int_{1}^{2}\left\vert F_{\rho ,\widetilde{\delta }%
,2^{k}r}^{\varphi _{a}}f\right\vert ^{2}dr\right) ^{1/2}\right\Vert
_{L^{p_{1}}(\mathbb{H}^{n})}\left\Vert \sup_{k\in
\mathbb{Z}
}\left( \sum_{\rho \in \widetilde{\delta }%
\mathbb{Z}
\cap \lbrack 0,2]}\int_{1}^{2}\left\vert F_{\sigma -\rho ,\widetilde{\delta }%
,2^{k}r}^{\varphi _{b}}g\right\vert ^{2}dr\right) ^{1/2}\right\Vert
_{L^{p_{2}}(\mathbb{H}^{n})}.  \label{Lp}
\end{eqnarray}%
Notice that $\sigma -\rho \in \widetilde{\delta }%
\mathbb{Z}
\cap \lbrack -4\delta -1,1+2\delta ]$ for any $\sigma \in \widetilde{\delta }%
\mathbb{Z}
\cap \lbrack 1-4\delta ,1+2\delta ]$, $\rho \in \widetilde{\delta }%
\mathbb{Z}
\cap \lbrack 0,2]$ and $F_{\sigma -\rho ,\widetilde{\delta },2^{k}}^{\varphi
_{b}}g=0$ if $\sigma -\rho \in \widetilde{\delta }%
\mathbb{Z}
\cap \lbrack -4\delta -1,0]$. So,
\begin{equation*}
\sup_{k\in
\mathbb{Z}
}\left( \sum_{\rho \in \widetilde{\delta }%
\mathbb{Z}
\cap \lbrack 0,2]}\int_{1}^{2}\left\vert F_{\rho ,\widetilde{\delta }%
,2^{k}r}^{\varphi _{a}}f\right\vert ^{2}dr\right) ^{1/2}=\sup_{k\in
\mathbb{Z}
}\left\vert D_{\widetilde{\delta },k}^{\varphi _{a}}f\right\vert ,\text{ }
\end{equation*}%
and
\begin{equation*}
\sup_{k\in
\mathbb{Z}
}\left( \sum_{\rho \in \widetilde{\delta }%
\mathbb{Z}
\cap \lbrack 0,2]}\int_{1}^{2}\left\vert F_{\sigma -\rho ,\widetilde{\delta }%
,2^{k}}^{\varphi _{b}}g\right\vert ^{2}dr\right) ^{1/2}\leq \sup_{k\in
\mathbb{Z}
}\left( \sum_{\sigma -\rho \in \widetilde{\delta }%
\mathbb{Z}
\cap \lbrack 0,2]}\int_{1}^{2}\left\vert F_{\sigma -\rho ,\widetilde{\delta }%
,2^{k}}^{\varphi _{b}}g\right\vert ^{2}dr\right) ^{1/2}=\sup_{k\in
\mathbb{Z}
}\left\vert D_{\widetilde{\delta },k}^{\varphi _{b}}g\right\vert .
\end{equation*}%
Using (\ref{Lp}) and Corollary \ref{Coro}, we see that
\begin{eqnarray*}
&&\left\Vert \sup_{k\in
\mathbb{Z}
}\int_{1}^{2}\left\vert S_{2^{k}r}^{\delta }(f,g)\right\vert dr\right\Vert
_{L^{p}(\mathbb{H}^{n})} \\
&\leq &\sum_{N=0}^{\infty }\frac{1}{N!}\sum_{0\leq a+b\leq N}c_{a,b}\delta
^{\kappa (a+b)}\sum_{\sigma \in \widetilde{\delta }%
\mathbb{Z}
k 1-4\delta ,1+2\delta ]}\left\vert \psi ^{(a+b)}\left( \frac{1-\sigma }{%
\delta }\right) \right\vert \\
&&\times \left\Vert \sup_{k\in
\mathbb{Z}
}\left\vert D_{\widetilde{\delta },k}^{\varphi _{a}}f\right\vert \right\Vert
_{L^{p_{1}}(\mathbb{H}^{n})}\left\Vert \sup_{k\in
\mathbb{Z}
}\left\vert D_{\widetilde{\delta },k}^{\varphi _{b}}g\right\vert \right\Vert
_{L^{p_{2}}(\mathbb{H}^{n})} \\
&\leq &C\delta ^{-\left[ (b-\frac{1}{2})(1-\frac{2}{p_{1}})-\frac{1}{p_{1}}%
\right] -\left[ (b-\frac{1}{2})(1-\frac{2}{p_{2}})-\frac{1}{p_{2}}\right]
-\varepsilon }\left\Vert f\right\Vert _{L^{p_{1}}(\mathbb{H}^{n})}\left\Vert
g\right\Vert _{L^{p_{2}}(\mathbb{H}^{n})} \\
&&\times \sum_{N=0}^{\infty }\frac{1}{N!}\sum_{0\leq a+b\leq N}c_{a,b}\delta
^{\kappa (a+b)}\sum_{\sigma \in \widetilde{\delta }%
\mathbb{Z}
\cap \lbrack 1-4\delta ,1+2\delta ]}\left\vert \psi ^{(a+b)}\left( \frac{%
1-\sigma }{\delta }\right) \right\vert \\
&\leq &C\delta ^{-(2b-1)(1-\frac{1}{p})-\frac{1}{p}-\varepsilon }\left\Vert
f\right\Vert _{L^{p_{1}}(\mathbb{H}^{n})}\left\Vert g\right\Vert _{L^{p_{2}}(%
\mathbb{H}^{n})}\delta ^{-1-\kappa }\sum_{N=0}^{\infty }\frac{1}{N!}%
\sum_{0\leq a+b\leq N}c_{a,b}\delta ^{(\kappa -1)(a+b)} \\
&\leq &C\delta ^{-(2b-1)(1-\frac{1}{p})-\frac{1}{p}-2-\varepsilon
}\left\Vert f\right\Vert _{L^{p_{1}}(\mathbb{H}^{n})}\left\Vert g\right\Vert
_{L^{p_{2}}(\mathbb{H}^{n})}
\end{eqnarray*}%
holds for $\kappa =1+\varepsilon $ with any $\varepsilon >0$ and $b>\frac{1}{%
2}(D-1)$ with $D=2n+1$ is the topological dimension of $\mathbb{H}^{n}$. It
follows that
\begin{eqnarray*}
\left\Vert S_{\ast }^{\delta }(f,g)\right\Vert _{L^{p}(\mathbb{H}^{n})}
&\leq &\left\Vert \sup_{k\in
\mathbb{Z}
}\int_{1}^{2}\left\vert S_{2^{k}r}^{\delta }(f,g)\right\vert dr\right\Vert
_{L^{p}(\mathbb{H}^{n})}+\left\Vert \sup_{k\in
\mathbb{Z}
}\int_{1}^{2}\left\vert \frac{\partial }{\partial r}S_{2^{k}r}^{\delta
}(f,g)\right\vert dr\right\Vert _{L^{p}(\mathbb{H}^{n})} \\
&\leq &C\left( \delta ^{-(2b-1)(1-\frac{1}{p})-\frac{1}{p}-2-\varepsilon
}+\delta ^{-(2b-1)(1-\frac{1}{p})-\frac{1}{p}-2-\varepsilon -1}\right)
\left\Vert f\right\Vert _{L^{p_{1}}(\mathbb{H}^{n})}\left\Vert g\right\Vert
_{L^{p_{2}}(\mathbb{H}^{n})} \\
&\leq &C\delta ^{-(2b-1)(1-\frac{1}{p})-\frac{1}{p}-2-\varepsilon
}\left\Vert f\right\Vert _{L^{p_{1}}(\mathbb{H}^{n})}\left\Vert g\right\Vert
_{L^{p_{1}}(\mathbb{H}^{n})}.
\end{eqnarray*}%
Therefore, whenever $\alpha >\alpha (p_{1},p_{2})=$ $D(1-\frac{1}{p})+\frac{3%
}{p}$, we can choose $b>\frac{1}{2}(D-1)$ and $\varepsilon >0$ such that $%
\alpha >(2b-1)(1-\frac{1}{p})+\frac{1}{p}+2+\varepsilon $. Then
\begin{eqnarray*}
\left\Vert S_{\ast }^{\alpha }(f,g)\right\Vert _{L^{p}(\mathbb{H}^{n})}
&\leq &\sum_{\delta \in D}\delta ^{\alpha }\left\Vert S_{\ast }^{\delta
}(f,g)\right\Vert _{L^{p}(\mathbb{H}^{n})}+\left\Vert S_{\ast
}^{0}(f,g)\right\Vert _{L^{p}(\mathbb{H}^{n})} \\
&\leq &\sum_{\delta \in D}\delta ^{\alpha }{}^{-\left( (2b-1)(1-\frac{1}{p})+%
\frac{1}{p}+\varepsilon \right) }\left\Vert f\right\Vert _{L^{p_{1}}(\mathbb{%
H}^{n})}\left\Vert g\right\Vert _{L^{p_{1}}(\mathbb{H}^{n})}+C\left\Vert
f\right\Vert _{L^{p_{1}}(\mathbb{H}^{n})}\left\Vert g\right\Vert _{L^{p_{2}}(%
\mathbb{H}^{n})} \\
&\leq &C\left\Vert f\right\Vert _{L^{p_{1}}(\mathbb{H}^{n})}\left\Vert
g\right\Vert _{L^{p_{2}}(\mathbb{H}^{n})}.
\end{eqnarray*}%
The proof of Theorem \ref{mainTh} is complete.
\end{proof}

\textbf{Acknowledgements} {\quad The first author is supported by National
Natural Science Foundation of China (Grant No. 12201584). }


\begin{thebibliography}{99}

\bibitem{Carbery} A. Carbery, J. L. Rubio de Francia and L. Vega, \textit{%
Almost everywhere summability of Fourier integrals, }J. London Math. Soc.
(2) 38 (1988) 513--524.

\bibitem{Gorges} D. Gorges and D. M\"{u}ller, Almost everywhere convergence
of Bochner-Riesz means on the Heisenberg group and fractional integration on
the dual, Proc. London Math. Soc. (3) 85 (2002), no. 1, 139--167.

\bibitem{L-W} H. Liu and M. Wang, \textit{Bilinear Riesz means on the
Heisenberg group, }Sci China Math. 62 (2019), no. 12, 2535--2556.

\bibitem{J-L} E. Jeong and S. Lee, \textit{Maximal estimates for the bilinear spherical averages and the bilinear Bochner-Riesz operators. }, J. Funct. Anal.  279 (2020), no.7,

\bibitem{GHH} L. Grafakos, D. He and P. Honz\'{\i}k, \textit{Maximal
operators associated with bilinear multipliers of limited decay}, J. Anal. Math. 143 (2021), no.1, 231--251.


\bibitem{Mull} D. M\"{u}ller, \textit{A restriction theorem for the
Heisenberg group}, Ann. Math (2). 131 (1990), no. 3, 567--587.

\bibitem{Mull2} D. M\"{u}ller, \textit{On Riesz means of eigenfunction
expansions for the Kohn-Laplacian}, J. Reine Angew. Math. 401 (1989),
113--121.

\bibitem{Strich} R. S. Strichartz, \textit{Harmonic analysis as spectral
theory of Laplacians}, J. Funct. Anal. 87 (1989), no. 1, 51--148.

\bibitem{Strich2} R. S. Strichartz, $L^{p}$\textit{\ harmonic analysis and
Radon transform on the Heisenberg group}, J. Funct. Anal. (96) 1991, no.2,
350--406.

\bibitem{Tao} T. Tao, \textit{On the maximal Bochner-Riesz conjecture in the
plane for} $p<2$,  Trans. Amer. Math. Soc. 354 (2002), no. 5, 1947--1959.

\bibitem{Thang} S. Thangavelu, Harmonic analysis on the Heisenberg group.
Progress in Mathematics, 159.
\end{thebibliography}
\end{document}